\documentclass[11pt]{amsart}
\usepackage{amsmath,amssymb,amsthm}
\usepackage{hyperref}
\usepackage{tikz}

\newtheorem{theorem}{Theorem}[section]
\newtheorem{lemma}[theorem]{Lemma}
\newtheorem{proposition}[theorem]{Proposition}
\newtheorem{corollary}[theorem]{Corollary}
\theoremstyle{definition}
\newtheorem{definition}[theorem]{Definition}
\newtheorem*{unnumbereddefinition}{Definition}
\newtheorem{remark}[theorem]{Remark}

\newtheorem{question}[theorem]{Question}

\usetikzlibrary{arrows.meta,calc,positioning,fit}

\tikzset{
  >={Latex[length=2.5mm]},
  knot/.style={circle,draw,thick,minimum size=8mm,inner sep=0pt,fill=white},
  surface/.style={draw,thick,fill=gray!12},
  carried/.style={draw=red!70!black,very thick},
  carriedb/.style={draw=blue!70!black,very thick},
  dashedarrow/.style={draw=red!70!black,dashed,->,thick},
  info/.style={font=\small},
}

\title[Amicable Knots on Minimal Genus Seifert Surfaces]{Amicable Knots on Minimal Genus Seifert Surfaces}
\author{Makoto Ozawa}
\address{Department of Natural Sciences, Faculty of Arts and Sciences, Komazawa University, 1-23-1 Komazawa, Setagaya-ku, Tokyo, 154-8525, Japan}
\email{w3c@komazawa-u.ac.jp}
\date{July 15, 2026}

\begin{document}

\begin{abstract}
For a knot $K\subset S^3$, let $S(K)$ denote the set of non-trivial knot types represented by
simple closed curves on a minimal genus Seifert surface of $K$.
We study the relation $J\in S(K)$ and its symmetric part, which leads to the notion of
\emph{amicable knots}: knots $K$ and $J$ are called amicable if each is represented by a
simple closed curve on a minimal genus Seifert surface of the other.

A classical result of Lyon implies that the family of torus knots is universal for this
realization problem: for every non-trivial knot type $J$, there exists a torus knot $T$
such that $J\in S(T)$.
In contrast, one of the main results of this paper is that no single knot is universal:
for every knot $K$, there exists a knot $J$ such that $J\notin S(K)$.

We also study explicit examples, keeping track of chirality throughout.
Writing $3_1^+=T(2,3)$ and $8_{19}^+=T(3,4)$ for the right-handed positive torus knots,
we show that $3_1^+$ and $8_{19}^+$ are amicable, whereas $3_1^+$ and the figure-eight
knot $4_1$ are not.  We also describe the hosting sets of both chiralities of the trefoil
in terms of primitive slope classes on their once-punctured torus fibers.
\end{abstract}

\maketitle

\section{Introduction}

Seifert surfaces are among the most basic objects in knot theory.
Minimal genus Seifert surfaces are especially important, since they reflect both the genus
of the knot and the topology of its complement.
In this paper, we study the following elementary but apparently unexplored question:
which knot types can be represented by simple closed curves on a minimal genus Seifert
surface of a given knot?

For a knot $K$, let $S(K)$ denote the set of non-trivial knot types represented by simple
closed curves on some minimal genus Seifert surface of $K$.
Thus, if $J\in S(K)$, then $J$ is realizable on a minimal genus Seifert surface of $K$.
The symmetric part of this relation is the main object of the paper.
We say that two knots $K$ and $J$ are \emph{amicable} if
\[
J\in S(K)\quad\text{and}\quad K\in S(J).
\]
Since $K\in S(K)$ for every knot $K$, the interesting question is the existence of
\emph{proper} amicable knots.

\begin{question}\label{que:proper-amicable}
Does every knot have a proper amicable knot?
\end{question}

\begin{question}\label{que:amicability-connected}
Is the graph of proper amicable pairs connected?
\end{question}

For chirality-sensitive examples, we use the notation
\[
3_1^+=T(2,3),\qquad 3_1^-=\overline{T(2,3)},
\qquad
8_{19}^+=T(3,4),\qquad 8_{19}^-=\overline{T(3,4)}.
\]
Thus the superscript records the handedness, and $T(p,q)$ with $p,q>0$ denotes the
right-handed positive torus knot.  The figure-eight knot $4_1$ is amphichiral, so no
superscript is needed for it.

A classical theorem of Lyon \cite{Lyon1980} shows that torus knots play a distinguished role:
for every non-trivial knot $J$, there exists a torus knot $T$ such that $J\in S(T)$.
Thus the family of torus knots is universal for the realization problem.
One of the main results of this paper is that this phenomenon is genuinely collective.

\begin{theorem}\label{thm:no-universal-host}
For every knot $K$, there exists a knot $J$ such that $J\notin S(K)$.
In particular, no knot is universal for the realization problem on minimal genus Seifert surfaces.
\end{theorem}

We also study explicit examples.
The trefoil has a unique minimal genus Seifert surface up to isotopy, and this surface is
a once-punctured torus.  We describe the hosting sets of $3_1^-$ and $3_1^+$ in terms of
primitive slope classes.  We then use a diagrammatic construction on the trefoil fiber to
show
\[
8_{19}^+\in S(3_1^+),
\]
and use the torus-knot inclusion proved later to obtain the reverse relation.
We also prove
\[
4_1\notin S(3_1^+),
\qquad
3_1^+\in S(4_1),
\]
so realizability need not be symmetric.

\begin{figure}[tb]
\centering
\begin{tikzpicture}[
    >=Latex,
    every node/.style={font=\small},
    knot/.style={draw, circle, minimum size=11mm, inner sep=1pt},
]

\node[knot] (T23) at (0,0) {$T(2,3)$};
\node[knot] (Tpq) at (3.8,0) {$T(p,q)$};
\node[knot] (J)   at (8.3,0) {$J$};

\draw[->, thick] (T23) -- (Tpq);
\draw[->, thick] (Tpq) -- (J);

\node[align=center] at ($(T23)!0.5!(Tpq)+(0,-1.1)$)
  {$S(T(2,3))\subset S(T(p,q))$\\
   \text{for suitable }(p,q)};

\node[align=center] at ($(Tpq)!0.55!(J)+(0.6,1.0)$)
{$J\in S(T(p,q))$\\for some $(p,q)$};

\node[
    draw,
    rounded corners,
    fit=(T23)(Tpq),
    inner xsep=10mm,
    inner ysep=10mm
] (box) {};

\node[above=1mm of box] {torus knots};

\end{tikzpicture}

\caption{A universal family versus the nonexistence of a universal knot.
Lyon's theorem says that every non-trivial knot is realized on some torus knot,
whereas Theorem~\ref{thm:no-universal-host} shows that no single knot realizes all knots.}
\label{fig:universal-family}
\end{figure}

The paper is organized as follows.
In Section~\ref{sec:definitions}, we define $S(K)$ and amicability.
In Section~\ref{sec:universal-hosting}, we recall Lyon's theorem and show that no single
knot is universal.
In Section~\ref{sec:trefoil}, we describe the hosting sets of the two trefoils and give a
careful realization of $T(3,4)$ on the fiber of $T(2,3)$.
In Section~\ref{sec:amicability-examples}, we prove a monotonicity property for torus knots
and give explicit amicable and non-amicable pairs.
We end with two questions motivated by these examples.

\section{Knots Realizable on Minimal Genus Seifert Surfaces}\label{sec:definitions}

In this section, we introduce the basic notions used throughout the paper.
All knots are considered in $S^3$ up to ambient isotopy, and $\mathcal{K}$ denotes
the set of all non-trivial knot types in $S^3$.

\subsection{Knots on minimal genus Seifert surfaces}

We begin by defining the set of knot types carried by minimal genus Seifert surfaces.

\begin{definition}\label{def:SK}
For a knot $K\in \mathcal{K}$, let $S(K)$ denote the set of all non-trivial knot types
represented by simple closed curves on some minimal genus Seifert surface of $K$.
Equivalently,
$J\in S(K)$
if and only if there exist a minimal genus Seifert surface $F$ for $K$ and a simple
closed curve $c\subset F$ such that $c$ represents the knot type $J$ in $S^3$.
\end{definition}

Thus we obtain a set-valued map \(S:\mathcal{K}\to 2^{\mathcal{K}}\).

\begin{remark}
We do not require the curve $c\subset F$ to be essential in $F$.
Hence inessential curves are allowed, and in particular the trivial knot can occur on
every minimal genus Seifert surface.
To avoid this ubiquitous but uninformative case, we restrict $S(K)$ to its non-trivial part.
Thus $S(K)\subset \mathcal{K}$.
\end{remark}

\begin{remark}
By definition, $S(K)$ depends only on the knot type of $K$.
Indeed, if $K$ and $K'$ are ambient isotopic, then any ambient isotopy carrying $K$ to $K'$
also carries minimal genus Seifert surfaces of $K$ to minimal genus Seifert surfaces of $K'$,
preserving the knot types represented by simple closed curves on them.
\end{remark}

\begin{proposition}\label{prop:reflexive}
For every knot \(K\in\mathcal K\), we have \(K\in S(K)\).
\end{proposition}

\begin{proof}
Let $F$ be a minimal genus Seifert surface for $K$.
A simple closed curve on $F$ parallel to $\partial F$ represents the same knot type as $K$ in $S^3$.
Hence $K\in S(K)$.
\end{proof}

\begin{proposition}[Mirror equivariance]\label{prop:mirror-equivariance}
For every knot $K$,
\[
S(\overline K)=\{\overline J\mid J\in S(K)\}.
\]
\end{proposition}

\begin{proof}
Reflection of $S^3$ carries a minimal genus Seifert surface of $K$ to a minimal genus
Seifert surface of $\overline K$ and carries every simple closed curve on the former to
the mirror of the corresponding knot on the latter.  Applying the same reflection twice
gives the reverse inclusion.
\end{proof}

\begin{proposition}\label{prop:connected-sum-hosting}
If \(J\in S(K)\) and \(J'\in S(K')\), then \(J\#J'\in S(K\#K')\).
\end{proposition}

\begin{proof}
Choose minimal genus Seifert surfaces $F$ for $K$ and $F'$ for $K'$ such that
$J\subset \operatorname{int}(F)$ and $J'\subset \operatorname{int}(F')$ are represented by simple closed curves.
Form the boundary connected sum $F\natural F'$ by attaching a band
\[
b\cong I\times I
\]
between small boundary arcs of $F$ and $F'$, chosen disjoint from $J\cup J'$.
Then $F\natural F'$ is a Seifert surface for $K\#K'$, and by Schubert's
additivity theorem for knot genus under connected sum
\cite[\S~5, Satz~2, p.~156]{Schubert1953}, it is a minimal genus Seifert
surface for $K\#K'$.

Choose embedded arcs $\lambda\subset F$ and $\lambda'\subset F'$ such that
$\lambda\cap J$ and $\lambda'\cap J'$ consist of one endpoint each, the other endpoints of
$\lambda$ and $\lambda'$ lie on the attaching arcs of the band $b$, and the interiors of
$\lambda$ and $\lambda'$ are disjoint from $J\cup J'$.
Let $A\subset F\natural F'$ be a narrow regular neighborhood in the surface of
\[
J\cup \lambda \cup b \cup \lambda' \cup J'.
\]
Since $F\natural F'$ is orientable, $A$ is an embedded pair of pants whose two boundary
components are parallel to $J$ and $J'$, and whose third boundary component $c$ is a
simple closed curve on $F\natural F'$.

The arcs $\lambda$, $\lambda'$ and the band $b$ may be chosen inside a small $3$-ball that
meets $J$ and $J'$ in trivial boundary-parallel subarcs. Since this $3$-ball is standard and
is disjoint from the rest of $J\cup J'$, the band $b$ is unknotted in the usual sense.
Therefore the third boundary component $c$ is exactly the standard band sum of $J$ and $J'$
along an unknotted band, and hence represents the connected sum $J\#J'$ in $S^3$.
Thus $J\#J'\in S(K\#K')$.
\end{proof}

\begin{definition}[Realization and amicability]\label{def:hosting-amicability}
For knots \(K,J\in\mathcal K\), we write
\[
K\to J
\]
if \(J\in S(K)\). When this holds, we say that \(K\) \emph{hosts} \(J\),
or equivalently that \(J\) is realizable on a minimal genus Seifert surface of \(K\).
We write
\[
K\leftrightarrow J
\]
if both \(K\to J\) and \(J\to K\) hold, and in this case we say that
\(K\) and \(J\) are \emph{amicable}. A knot \(J\) is a \emph{proper amicable knot}
of \(K\) if, in addition, \(J\neq K\). We use \(K\nrightarrow J\) and
\(K\not\leftrightarrow J\) for the corresponding negations.
\end{definition}

\begin{corollary}\label{cor:connected-sum-amicability}
If \(K\leftrightarrow J\) and \(K'\leftrightarrow J'\), then \(K\#K'\leftrightarrow J\#J'\).
\end{corollary}

\begin{proof}
This follows immediately from Proposition~\ref{prop:connected-sum-hosting}.
\end{proof}

\begin{corollary}\label{cor:stabilized-amicability}
If \(K\leftrightarrow J\), then for every knot \(L\) one has
\[
K\#L \leftrightarrow J\#L.
\]
\end{corollary}

\begin{proof}
Since \(L\in S(L)\) by Proposition~\ref{prop:reflexive}, the knot \(L\) is amicable to itself.
Now apply Corollary~\ref{cor:connected-sum-amicability} with \(K'=J'=L\).
\end{proof}

\begin{remark}
From the viewpoint of the amicability graph, Corollaries~\ref{cor:connected-sum-amicability}
and~\ref{cor:stabilized-amicability} show that connected sum preserves edges and produces
infinite families from any single proper amicable pair.
Thus once one nontrivial edge is known, one immediately obtains many others by taking
connected sums with arbitrary knots.
This observation may be relevant to Question~\ref{que:amicability-connected}, although it does not by itself imply that
the amicability graph is connected.
\end{remark}

\subsection{The realization relation and amicable knots}

Definition~\ref{def:hosting-amicability} makes $S(K)$ into a directed relation on knots.

\begin{remark}
In general, the relation $K\to J$ is not symmetric.
Thus realizability on minimal genus Seifert surfaces should be viewed as a genuinely directed relation on $\mathcal{K}$,
while amicability is its symmetric part.
\end{remark}

\begin{remark}
Amicability is symmetric by definition, but there is no reason for it to be transitive.
Hence amicability is not, in general, an equivalence relation.
\end{remark}

\begin{remark}
Since $K\in S(K)$ for every knot $K$, each knot is amicable to itself.
Thus the existence of an amicable knot is trivial unless one asks for a proper one.
\end{remark}

\begin{definition}\label{def:amicability-graph}
The \emph{amicability graph} is the undirected graph whose vertex set is $\mathcal{K}$,
and where two distinct vertices $K$ and $J$ are joined by an edge if and only if \(J\in S(K)\) and \(K\in S(J)\).
\end{definition}

\begin{remark}
By excluding the case $K=J$ in the definition of the amicability graph, we regard the graph
as recording non-trivial mutual realizability.
The reflexive relation $K\in S(K)$ is still present at the level of the relation itself,
but is not drawn as a loop in the graph.
\end{remark}

\section{Universal Families and the Nonexistence of a Universal Knot}\label{sec:universal-hosting}

In this section, we study universal phenomena for realization on minimal genus Seifert surfaces.
We first recall Lyon's theorem, which implies that torus knots form a universal host family.
We then show that this universality is genuinely collective:
although every non-trivial knot is hosted by some torus knot, no single knot realizes all knot types.
The proof uses a uniform complexity bound for knots hosted by a fixed knot, expressed in terms
of a Heegaard-surface invariant and tunnel number.

\subsection{Lyon's theorem and torus knots}

A classical theorem of Lyon shows that torus knots play a distinguished role in this realization problem.

\begin{theorem}[Lyon]\label{thm:lyon}
For every non-trivial knot $J$ in $S^3$, there exists a torus knot $T$ such that \(J\in S(T)\).
\end{theorem}

\begin{proof}
This is exactly the main result of Lyon \cite{Lyon1980}, reformulated in the present terminology.
\end{proof}

\begin{corollary}\label{cor:torus-universal}
The family of torus knots is universal for this realization problem.
\end{corollary}

\begin{proof}
This is an immediate reformulation of Theorem~\ref{thm:lyon}.
\end{proof}

Thus universality does occur at the level of families.
Our next goal is to show that no individual knot is universal.

\subsection{A Heegaard-surface invariant and tunnel number}

To detect restrictions on the knots hosted by a fixed knot, we introduce a simple
Heegaard-theoretic invariant.

\begin{definition}[\cite{Morimoto1994}]\label{def:hJ}
For a knot $J$ in $S^3$, define
\[
h(J)
=
\min\left\{
g(\Sigma)
\;\middle|\;
\Sigma \subset S^3 \text{ is a Heegaard surface containing } J
\right\}.
\]
Thus $h(J)$ is the minimal genus of a Heegaard surface of $S^3$ on which $J$ can be realized.
\end{definition}

The invariant $h(J)$ dominates the tunnel number.

\begin{proposition}[\cite{Morimoto1994}]\label{prop:tunnel-le-h}
For every knot $J$ in $S^3$,
\[
t(J)\le h(J).
\]
\end{proposition}

\begin{proof}
Suppose that $J$ lies on a Heegaard surface $\Sigma$ of genus $g$.
Using one of the two handlebodies bounded by $\Sigma$, one obtains a tunnel system for $J$
with at most $g$ tunnels.
Hence
\[
t(J)\le g.
\]
Taking the minimum over all such Heegaard surfaces gives
\[
t(J)\le h(J).
\]
\end{proof}

\subsection{Hosted knots and complexity bounds}

We now show that if a knot $J$ is hosted by $K$, then the complexity of $J$ is bounded
in terms of a minimal genus Seifert surface of $K$.

Let $F\subset S^3$ be a compact connected orientable surface of positive genus with one boundary component,
and let
$N(F)$
be a regular neighborhood of $F$ in $S^3$. Put
\[
E(F):=\operatorname{cl}(S^3\setminus N(F)).
\]
Then $N(F)$ is a handlebody of genus $2g(F)$, and $\partial E(F)=\partial N(F)$ is a closed orientable surface of genus $2g(F)$.

\begin{definition}\label{def:deltaF}
Let $F\subset S^3$ be a compact connected orientable surface of positive genus with one boundary component.
Define
\[
\delta(F):=g_H(E(F))-2g(F),
\]
where $g_H(E(F))$ denotes the Heegaard genus of the compact $3$-manifold $E(F)$.
\end{definition}

\begin{remark}
The manifold $E(F)$ is a compact orientable $3$-manifold with connected boundary
$\partial E(F)=\partial N(F)$,
which is a closed surface of genus $2g(F)$. Hence every Heegaard surface for $E(F)$ has genus at least $2g(F)$, so \(g_H(E(F))\ge 2g(F)\). Therefore \(\delta(F)\ge 0\). Moreover, $\delta(F)=0$
if and only if $E(F)$ is a handlebody, equivalently, if and only if $\partial N(F)$ itself is a
Heegaard surface of $S^3$.
\end{remark}

We recall the convention for compression bodies that will be used below.

\begin{definition}\label{def:compression-body}
A \emph{compression body} \(V\) is either a handlebody, in which case
\(\partial_-V=\varnothing\), or a compact connected orientable \(3\)-manifold
obtained from \(F\times I\), where \(F\) is a nonempty closed orientable surface,
possibly disconnected, with no \(2\)-sphere components, by attaching \(1\)-handles
to \(F\times\{1\}\). In the latter case, the \emph{negative boundary}
is
\[
\partial_-V=F\times\{0\},
\]
and the \emph{positive boundary} is
\[
\partial_+V=\operatorname{cl}\bigl(\partial V\setminus\partial_-V\bigr).
\]
Thus \(\partial_+V\) is the boundary obtained from \(F\times\{1\}\) after the
\(1\)-handles have been attached.
\end{definition}

The following elementary lemma is the key step.

\begin{lemma}\label{lem:push-to-heegaard}
Let $M$ be a compact orientable $3$-manifold with connected boundary, and let
$M=V\cup_{\Sigma} W$
be a Heegaard splitting such that $V$ is a compression body with
$\partial_-V=\partial M$.
Then every simple closed curve on $\partial M$ is isotopic in $V$ to a simple closed
curve on $\Sigma$.
\end{lemma}

\begin{proof}
By definition, the compression body $V$ is obtained from $\partial M\times I$ by attaching
$1$-handles to $\partial M\times\{1\}$. Let $c\subset \partial M=\partial M\times\{0\}$ be a
simple closed curve. First push $c$ through the product region to a simple closed curve
$c_1\subset \partial M\times\{1\}$. The positive boundary $\Sigma=\partial_+V$ is obtained from
$\partial M\times\{1\}$ by attaching $1$-handles. After a small isotopy in
$\partial M\times\{1\}$, we may assume that $c_1$ is disjoint from the attaching disks of these
$1$-handles. Then the inclusion \(\partial M\times\{1\}\setminus (\text{attaching disks})\hookrightarrow \Sigma\)
identifies $c_1$ with a simple closed curve on $\Sigma$. Thus $c$ is isotopic in $V$ to a
simple closed curve on $\Sigma$.
\end{proof}

We also need the following standard fact.

\begin{lemma}\label{lem:compressionbody-plus-handlebody}
Let $H$ be a handlebody, and let $V$ be a compression body with
$\partial_-V=\partial H$.
Then the manifold obtained by gluing $H$ and $V$ along $\partial H=\partial_-V$ is a
handlebody of genus $g(\partial_+V)$.
\end{lemma}

\begin{proof}
The compression body $V$ is obtained from $\partial_-V\times I$ by attaching $1$-handles
to $\partial_-V\times\{1\}$.
After gluing $H$ to $\partial_-V\times\{0\}=\partial H$, the product region
$\partial_-V\times I$ is absorbed into $H$, so the resulting manifold is obtained from the
handlebody $H$ by attaching finitely many $1$-handles.
Therefore the glued manifold is again a handlebody.
Its boundary is the positive boundary $\partial_+V$, and hence its genus is $g(\partial_+V)$.
\end{proof}

A Heegaard splitting of a compact orientable $3$-manifold $M$ with connected boundary is, by
convention, written in the form
\[
M=V\cup_{\Sigma}W,
\]
where $V$ is a compression body with $\partial_-V=\partial M$ and $W$ is a handlebody.

\begin{theorem}\label{thm:h-bound-from-F}
Let $F\subset S^3$ be a compact connected orientable surface of positive genus with one boundary component,
and let $J\subset F$ be a simple closed curve.
Then
\[
h(J)\le 2g(F)+\delta(F).
\]
\end{theorem}

\begin{proof}
Set
\[
N:=N(F), \qquad E:=E(F)=\operatorname{cl}(S^3\setminus N).
\]
Choose a Heegaard splitting of $E$ of minimal genus and write it as
\[
E=V\cup_{\Sigma} W,
\]
where $V$ is a compression body with
$\partial_-V=\partial E=\partial N$.
Then
\[
g(\Sigma)=g_H(E)=2g(F)+\delta(F).
\]

The curve $J\subset F$ determines a simple closed curve on $\partial N=\partial E$
representing the same knot type.
By Lemma~\ref{lem:push-to-heegaard}, this curve is isotopic in $V$ to a simple closed
curve $J'\subset\Sigma$.
Since this isotopy takes place inside the submanifold $V\subset S^3$, the isotopy extension
theorem (see, for example, \cite[Chapter~8]{Hirsch1976}) promotes it to an ambient isotopy of $S^3$.
Hence $J'$ represents the same knot type as $J$.

Now glue $N$ to $V$ along $\partial N=\partial_-V$.
By Lemma~\ref{lem:compressionbody-plus-handlebody}, the manifold $N\cup V$ is a handlebody
whose boundary is $\Sigma$.
Since $W$ is also a handlebody, we obtain a Heegaard splitting
\[
S^3=(N\cup V)\cup_{\Sigma}W
\]
of genus $g(\Sigma)=2g(F)+\delta(F)$.
By construction, the curve $J'$ lies on $\Sigma$ and represents the knot type $J$.
Therefore
\[
h(J)\le g(\Sigma)=2g(F)+\delta(F).
\]
\end{proof}

Applying this to a minimal genus Seifert surface gives the following.

\begin{corollary}\label{cor:h-bound-from-K}
Let $K$ be a knot, and let $F$ be a minimal genus Seifert surface for $K$.
If $J\in S(K)$ is represented by a simple closed curve on $F$, then
\[
h(J)\le 2g(K)+\delta(F).
\]
\end{corollary}

\begin{proof}
Since $F$ is a minimal genus Seifert surface for $K$, one has
$g(F)=g(K)$.
Now apply Theorem~\ref{thm:h-bound-from-F}.
\end{proof}

To obtain a bound depending only on $K$, we record the worst-case value of $\delta(F)$
among all minimal genus Seifert surfaces of $K$.

\begin{definition}\label{def:DeltaK}
For a knot $K$, define
\[
\Delta(K)
:=
\sup\left\{
\delta(F)
\;\middle|\;
F \text{ is a minimal genus Seifert surface for } K
\right\}.
\]
\end{definition}

\begin{proposition}\label{prop:Delta-finite}
For every knot $K$, the quantity $\Delta(K)$ is finite.
\end{proposition}

\begin{proof}
Every minimal genus Seifert surface for $K$ is incompressible.  Wilson's
normal-surface theorem gives finitely many incompressible Seifert surfaces
$S_1,\dots,S_r$ and finitely many closed incompressible surfaces
$Q_1,\dots,Q_s$ in $E(K)$ such that every incompressible Seifert surface is
isotopic to a Haken sum
\[
S_i+\sum_{j=1}^s a_j Q_j,
\qquad a_j\in\mathbb Z_{\ge 0};
\]
see \cite[Theorem~1.1]{Wilson2008}.

All minimal genus Seifert surfaces for $K$ have the same Euler characteristic
$1-2g(K)$.  Since each closed incompressible surface $Q_j$ has
$\chi(Q_j)\le 0$ and Euler characteristic is additive under Haken sum, the
coefficients of those $Q_j$ with $\chi(Q_j)<0$ are bounded.  There are
therefore only finitely many partial sums obtained by retaining the Seifert
surface $S_i$ and all negative-Euler-characteristic summands while omitting the
torus summands.

For each partial sum that occurs, consider the component containing the unique
boundary curve.  Its genus cannot be smaller than $g(K)$.  On the other hand,
the Euler characteristic of the entire partial sum is $1-2g(K)$, because the
omitted summands are tori.  It follows that the boundary-containing component
has genus exactly $g(K)$ and that every remaining closed component has Euler
characteristic zero.  Thus the latter components are tori and may be absorbed
into the torus summands.  We consequently obtain a finite collection of
reference minimal genus Seifert surfaces
\[
F_1,\dots,F_N
\]
such that every minimal genus Seifert surface $F$ is obtained from one of the
$F_i$ by Haken summing with parallel copies of finitely many incompressible
tori.

For completeness, we spell out the local relation between a torus Haken
sum and the usual spinning operation.  Let $T$ be one of the incompressible
torus summands.  After isotoping $F_i$ and $T$ into minimal position and
choosing a product neighborhood $T\times[-1,1]$, the intersection
\[
F_i\cap\bigl(T\times[-1,1]\bigr)
\]
is a union of vertical spanning annuli of a single slope.  The regular
exchange defining the Haken sum of $F_i$ with a parallel copy of $T$ is
precisely the image of these annuli under a Dehn twist of
$T\times[-1,1]$ that is the identity on $T\times\{-1,1\}$.  Thus this
Haken sum is carried to $F_i$ by a self-homeomorphism of $E(K)$ supported
in the product neighborhood.  Repeating the construction for parallel
copies and performing the finitely many torus summands successively gives a
self-homeomorphism
\[
\varphi\colon E(K)\longrightarrow E(K)
\]
which is the identity on $\partial E(K)$ and carries the relevant reference
surface $F_i$ to $F$.  This is the spinning action used in the description of
the Kakimizu complex; see \cite[Section~3 and Lemmas~12 and~15]{JohnsonPelayoWilson2014}.

Because $\varphi$ is the identity on $\partial E(K)$, it extends by the identity
over $N(K)$ to a self-homeomorphism of $S^3$.  Consequently it carries a
regular neighborhood of $F_i$ to one of $F$ and induces a homeomorphism
\[
E(F_i)\cong E(F).
\]
Therefore
\[
g_H(E(F))=g_H(E(F_i))
\quad\text{and}\quad
g(F)=g(F_i)=g(K),
\]
so that $\delta(F)=\delta(F_i)$.  Hence
\[
\bigl\{\delta(F)\mid F\text{ is a minimal genus Seifert surface for }K\bigr\}
\subset
\{\delta(F_1),\dots,\delta(F_N)\},
\]
and the set on the left is finite.  Thus
\[
\Delta(K)=\max_{1\le i\le N}\delta(F_i)<\infty.
\]
\end{proof}

\begin{remark}
In particular, the supremum in Definition~\ref{def:DeltaK} is actually attained; that is,
$\Delta(K)$ is a maximum.
\end{remark}

\begin{remark}
It would be interesting to understand the possible values of \(\Delta(K)\) more concretely.
The definition is designed only to provide a uniform bound for hosted knots, and we do not
need an explicit example with \(\Delta(K)>0\) in the arguments below.
Nevertheless, computing \(\Delta(K)\) for concrete knot types, or deciding when it must vanish,
seems to be a natural problem in its own right.
\end{remark}

Although the minimal genus Seifert surface witnessing $J\in S(K)$ may depend on $J$,
Definition~\ref{def:DeltaK} provides a uniform bound because
$\delta(F)\le \Delta(K)$ for every such witnessing surface $F$.

\begin{theorem}\label{thm:h-uniform}
If $J\in S(K)$, then
\[
h(J)\le 2g(K)+\Delta(K).
\]
\end{theorem}

\begin{proof}
Since $J\in S(K)$, there exists a minimal genus Seifert surface $F$ for $K$ such that
$J$ is represented by a simple closed curve on $F$.
By Corollary~\ref{cor:h-bound-from-K},
\[
h(J)\le 2g(K)+\delta(F)\le 2g(K)+\Delta(K).
\]
\end{proof}

\begin{corollary}\label{cor:tunnel-bound}
If $J\in S(K)$, then
\[
t(J)\le h(J)\le 2g(K)+\Delta(K).
\]
\end{corollary}

\begin{proof}
Combine Proposition~\ref{prop:tunnel-le-h} with Theorem~\ref{thm:h-uniform}.
\end{proof}

Thus the class of knots hosted by a fixed knot is subject to a uniform complexity bound.

\subsection{No universal host knot}

We now prove the main negative result stated in the introduction.

\begin{proof}[Proof of Theorem~\ref{thm:no-universal-host}]
Fix a knot $K$.
By Corollary~\ref{cor:tunnel-bound}, every knot $J\in S(K)$ satisfies
\[
t(J)\le 2g(K)+\Delta(K).
\]
On the other hand, there exist knots in $S^3$ with arbitrarily large tunnel number.
This follows, for example, from the theorem of Scharlemann and Schultens
that the tunnel number of the connected sum of $n$ non-trivial knots is at least $n$
\cite{ScharlemannSchultens1999}.
Choose a knot $J$ with
\[
t(J)>2g(K)+\Delta(K).
\]
Then necessarily
\[
J\notin S(K).
\]
Hence $K$ is not a universal host.
\end{proof}

\begin{remark}
Theorem~\ref{thm:no-universal-host} gives a sharp contrast with
Corollary~\ref{cor:torus-universal}.
Torus knots together form a universal host family, yet no single knot is universal.
Thus universality is a genuinely collective phenomenon.
\end{remark}

\section{The Hosting Sets of the Trefoils}\label{sec:trefoil}

In this section, we describe the hosting sets of the left- and right-handed trefoils.
The distinction between the two chiralities is important.  We first use Yamada's
canonical form on the left-handed trefoil fiber, where Baker's positive-braid formulas
apply directly, and then pass to the right-handed trefoil by mirror equivariance.
Finally, we give a separate diagrammatic realization of $T(n,n+1)$ on the right-handed
trefoil fiber.

Every minimal genus Seifert surface for a trefoil is incompressible.  By Tsau's
classification of incompressible surfaces in torus-knot exteriors
\cite{Tsau1994}, each trefoil has a unique minimal genus Seifert surface up to
isotopy.  Consequently, it is enough to work with the standard fiber surface.

\subsection{Yamada's slope basis and the hosting set}

Let $F^-=F^{+,+}$ be Yamada's genus-one fiber surface whose boundary is the
left-handed trefoil $3_1^-$ \cite[Definition~2.1 and Table~1]{Yamada2010}.
In Yamada's disk--band picture, $F^-$ is the union of a half-disk $D$ and two
bands $b_L$ and $b_R$.  Let
\[
a=c_L,\qquad b=c_R
\]
be the clockwise-oriented core circles of the left and right bands, respectively,
as in Yamada's Figure~3.  Orient $F^-$ so that
\[
[a]\cdot[b]=1.
\]
This explicitly fixes the ordered basis $([a],[b])$ of $H_1(F^-;\mathbb Z)$.

Every essential non-boundary-parallel simple closed curve on a once-punctured torus
represents a primitive homology class
\[
m[a]+n[b],\qquad \gcd(m,n)=1,
\]
and conversely every primitive class is represented by such a curve, uniquely up to
isotopy and reversal of orientation.

\begin{unnumbereddefinition}[The notation $K^-(m,n)$]
For coprime integers $m,n$, let $c_{m,n}\subset F^-$ be the essential simple closed
curve with oriented homology class $m[a]+n[b]$.  We define $K^-(m,n)$ to be the
unoriented knot type represented by $c_{m,n}$ in $S^3$.  This uppercase notation is
ours; in Yamada's notation the same curve is $k^{+,+}(m,n)$.
\end{unnumbereddefinition}

Yamada's Type~VII canonical-form theorem \cite[Theorem~1.1]{Yamada2010}, together
with Lemma~3.2 and Corollary~2.4 of that paper, gives a representative with
\[
0<m<n,\qquad \gcd(m,n)=1.
\]

\begin{proposition}\label{prop:trefoil-hosting-set}
One has
\[
S(3_1^-)
=
\{3_1^-\}
\cup
\{K^-(m,n)\mid 0<m<n,\ \gcd(m,n)=1\}.
\]
Consequently,
\[
S(3_1^+)
=
\{3_1^+\}
\cup
\{\overline{K^-(m,n)}\mid 0<m<n,\ \gcd(m,n)=1\}.
\]
\end{proposition}

\begin{proof}
The uniqueness of the minimal genus Seifert surface reduces the calculation to the
fiber $F^-$.  A non-trivial simple closed curve on $F^-$ is either boundary-parallel
or essential and non-boundary-parallel.  A boundary-parallel curve represents $3_1^-$.
For every other curve, Yamada's Theorem~1.1 gives the canonical form, and his
Lemma~3.2 and Corollary~2.4 reduce the parameters to the stated range.
The formula for $S(3_1^+)$ follows from Proposition~\ref{prop:mirror-equivariance}.
\end{proof}

\subsection{Genus and signature consequences}

\begin{proposition}[Baker]\label{prop:trefoil-Kmn-genus}
For $0<m<n$ and $\gcd(m,n)=1$,
\[
g(K^-(m,n))
=
\frac{m^2+n^2+mn-2m-2n+1}{2}.
\]
\end{proposition}

\begin{proof}
Baker proved that every essential knot on the fiber of the left-handed trefoil can be
represented as the closure of a positive braid; see
\cite[Theorem~B.0.1 and Appendix~B, Case~1]{Baker2004}.
In Appendix~B.2.2 he computes that the Seifert surface obtained from Seifert's algorithm
for the corresponding positive braid has Euler characteristic
\[
\chi=-m^2-mn-n^2+2m+2n.
\]
Since a positive braid diagram is homogeneous, Cromwell's theorem implies that this
Seifert surface has minimal genus \cite{Cromwell1989}.  The displayed formula follows.
\end{proof}

\begin{corollary}\label{cor:4_1_not_in_S3_1}
For both choices of chirality,
\[
4_1\notin S(3_1^\pm).
\]
\end{corollary}

\begin{proof}
A boundary-parallel curve on $F^-$ is the trefoil, not $4_1$.  Every essential
non-boundary-parallel knot on $F^-$ is the closure of a non-trivial positive braid
\cite[Appendix~B]{Baker2004}, and hence has nonzero signature
\cite{Rudolph1982}.  Since $\sigma(4_1)=0$, we have $4_1\notin S(3_1^-)$.
The figure-eight knot is amphichiral, so Proposition~\ref{prop:mirror-equivariance}
gives $4_1\notin S(3_1^+)$ as well.
\end{proof}

\begin{corollary}\label{cor:5_1_not_in_S3_1}
No knot of type $5_1$, of either chirality, belongs to $S(3_1^-)$ or $S(3_1^+)$.
\end{corollary}

\begin{proof}
A boundary-parallel curve has genus one.  If a curve in the second set of
Proposition~\ref{prop:trefoil-hosting-set} represented a knot of type $5_1$, then
Proposition~\ref{prop:trefoil-Kmn-genus} and
$g(5_1)=2$ would give
\[
\frac{m^2+n^2+mn-2m-2n+1}{2}=2
\]
for coprime integers $0<m<n$.  If $m=1$, this becomes $n(n-1)/2=2$, which has no
integer solution.  If $m\ge2$, then $n\ge3$, and the left-hand side is at least
\[
\frac{2^2+3^2+2\cdot3-2\cdot2-2\cdot3+1}{2}=5.
\]
This proves the assertion for $S(3_1^-)$.  The assertion for $S(3_1^+)$ follows by
mirroring, since genus is unchanged by mirror image.
\end{proof}

\subsection{The torus knots \texorpdfstring{$T(n,n+1)$}{T(n,n+1)} on the trefoil fiber}

We now record the concrete realization needed for amicability.  Let $F^+$ be the
fiber surface of the right-handed trefoil $3_1^+=T(2,3)$.  It is the plumbing of two
positive Hopf bands; see, for example, \cite[Figure~B.1]{Baker2004}.
Fix an ordered pair $x,y$ of oriented core curves of these Hopf bands.  We use the
Seifert-matrix convention
\[
V_{ij}=\operatorname{lk}(x_i,x_j^+),
\qquad (x_1,x_2)=(x,y),
\]
where $x_j^+$ denotes the positive normal push-off of $x_j$ from $F^+$.  We choose
the order and orientations so that
\[
V=
\begin{pmatrix}
1&0\\
-1&1
\end{pmatrix}.
\]
Let $h\colon F^+\to F^+$ be the monodromy, with mapping-torus convention
\[
S^3\setminus \operatorname{int}N(3_1^+)
\cong F^+\times[0,1]/(z,1)\sim(h(z),0),
\]
and represent classes in $H_1(F^+;\mathbb Z)$ by column vectors in the ordered
basis $([x],[y])$.  With these conventions, the standard relation between the
Seifert form and the monodromy is
\[
Vh_*=V^{\mathsf T}.
\]
Consequently,
\[
h_*=V^{-1}V^{\mathsf T}
=
\begin{pmatrix}
1&-1\\
1&0
\end{pmatrix}.
\]

\begin{proposition}\label{prop:Tn-on-trefoil}
For every integer $n\ge2$, the surface $F^+$ contains an essential simple closed curve
$\gamma_n$ representing the right-handed torus knot $T(n,n+1)$.  With the basis above,
\[
[\gamma_n]=n[x]-[y].
\]
In particular,
\[
8_{19}^+=T(3,4)\in S(3_1^+).
\]
The mirror construction gives $8_{19}^-\in S(3_1^-)$.
\end{proposition}

\begin{proof}
Vidussi draws the trefoil fiber as a plumbing of two twisted annuli and constructs a
simple closed curve passing through one annulus in $n$ parallel strands and through the
other in one strand \cite[Section~3, Figures~1--3]{Vidussi2006}.  In the basis fixed
above, the curve has class $n[x]-[y]$.  The crucial point is to track all strands through
the plumbing region as one braid, rather than treating the twists of the two Hopf bands
as independent full twists.  Vidussi's braid isotopy identifies the resulting curve with
$T(n,n+1)$.  His displayed convention is the mirror convention; reflecting the whole
construction gives the right-handed statement above.

For $n=3$, the strand tracking produces the four-strand braid
\[
\beta=\sigma_1\sigma_2\sigma_3(\sigma_1\sigma_2)^3.
\]
A cyclic conjugation followed by a positive Markov destabilization gives
\[
\widehat\beta
\cong
\widehat{(\sigma_1\sigma_2)^4\sigma_3}
\cong
\widehat{(\sigma_1\sigma_2)^4}
=T(3,4),
\]
which provides a direct check in the case used below.
\end{proof}

\begin{remark}
The braid calculation in the proof must be read from the whole plumbed surface.
Multiplying two formal full-twist contributions independently would instead produce the
wrong braid.  The explicit basis and Seifert matrix above remove the corresponding sign
and slope ambiguity.
\end{remark}

\section{Examples and Amicability Phenomena}\label{sec:amicability-examples}

In this section, we first prove a monotonicity property for torus-knot hosting sets and
then combine it with Proposition~\ref{prop:Tn-on-trefoil}.  We obtain a genuine amicable
pair with matching chirality and a contrasting asymmetric pair involving the figure-eight
knot.

\subsection{A preliminary inclusion for torus knots}

\begin{proposition}\label{prop:torus-inclusion}
Assume that $p,q,p',q'\ge2$ are integers satisfying
\[
\gcd(p,q)=\gcd(p',q')=1,
\qquad
p\le p',\quad q\le q'.
\]
Then
\[
S(T(p,q))\subset S(T(p',q')).
\]
\end{proposition}

\begin{proof}
Let $T\subset S^3$ be the standard unknotted torus and write
\[
S^3=V\cup_T W,
\]
where $V$ and $W$ are solid tori.  For coprime integers $r,s\ge2$, a minimal genus
Seifert surface $F_{r,s}$ of $T(r,s)$ is obtained from $r$ meridian disks in $V$ and
$s$ meridian disks in $W$ by smoothing their intersections; see \cite{Lyon1980}.

Choose pairwise disjoint meridian disks
\[
D_1,\dots,D_{p'}\subset V,
\qquad
E_1,\dots,E_{q'}\subset W
\]
in standard position so that $F_{p,q}$ uses the first $p$ and first $q$ disks, while
$F_{p',q'}$ uses all of them.  Let
\[
P=
\bigcup_{\substack{1\le i\le p\\ q<j\le q'}}(D_i\cap E_j)
\;\cup\;
\bigcup_{\substack{p<i\le p'\\ 1\le j\le q}}(D_i\cap E_j).
\]
This is a finite subset of $F_{p,q}$.  Outside arbitrarily small pairwise disjoint disk
neighborhoods of the points of $P$, the surfaces $F_{p,q}$ and $F_{p',q'}$ coincide as
subsets of $S^3$.

Let $J\in S(T(p,q))$ and choose a representing curve $c\subset F_{p,q}$.  A small
isotopy in $F_{p,q}$ moves $c$ away from the chosen disk neighborhoods of $P$ without
changing its knot type.  The resulting curve lies in the common part of the two surfaces,
and hence also lies on $F_{p',q'}$.  Therefore $J\in S(T(p',q'))$.
\end{proof}

\begin{remark}
Mirroring Proposition~\ref{prop:torus-inclusion} gives the analogous inclusion for the
left-handed torus knots.
\end{remark}

\begin{corollary}\label{cor:small-torus-on-large-torus}
Let $p,q\ge2$ and assume $\gcd(p,q)=\gcd(p+1,q+1)=1$.  Then
\[
T(p,q)\in S(T(p+1,q+1)).
\]
\end{corollary}

\begin{proof}
By Proposition~\ref{prop:reflexive}, $T(p,q)\in S(T(p,q))$.  Now apply
Proposition~\ref{prop:torus-inclusion}.
\end{proof}

\subsection{Amicability between \texorpdfstring{$3_1^+$ and $8_{19}^+$}{3-1+ and 8-19+}}

\begin{proposition}\label{prop:3_1-8_19-amicable}
The right-handed knots $3_1^+=T(2,3)$ and $8_{19}^+=T(3,4)$ are amicable.
Their mirror images $3_1^-$ and $8_{19}^-$ are also amicable.
\end{proposition}

\begin{proof}
Proposition~\ref{prop:Tn-on-trefoil} gives
\[
8_{19}^+=T(3,4)\in S(T(2,3))=S(3_1^+),
\]
so $3_1^+\to8_{19}^+$.  On the other hand,
Proposition~\ref{prop:torus-inclusion} gives
\[
S(T(2,3))\subset S(T(3,4)).
\]
Since $T(2,3)\in S(T(2,3))$, it follows that
\[
3_1^+=T(2,3)\in S(T(3,4))=S(8_{19}^+).
\]
Thus $8_{19}^+\to3_1^+$, and the right-handed pair is amicable.
The statement for the left-handed pair follows from
Proposition~\ref{prop:mirror-equivariance}.
\end{proof}

\subsection{A non-amicable example: the trefoil and the figure-eight knot}

\begin{proposition}\label{prop:3_1-in-S4_1}
Both chiralities of the trefoil lie in $S(4_1)$.
\end{proposition}

\begin{proof}
Use Yamada's figure-eight fiber $F^{-,+}=D\cup b_L\cup b_R$.
Let
\[
\alpha=c_L,\qquad \beta=c_R
\]
be the clockwise-oriented core circles of the two bands, as in
\cite[Definition~2.1 and Figure~3]{Yamada2010}.  Thus the slope $(p,q)$ means the
primitive class $p[\alpha]+q[\beta]$.

The slope $(1,2)$ curve is Yamada's $k^{-,+}(1,2)$.  Tracking this curve in the
disk--band diagram, equivalently specializing Baker's figure-eight train-track isotopy to
weights $(1,2)$ \cite[Figure~B.9]{Baker2004}, gives the closure of the two-strand braid
$\sigma_1^3$.  Hence it is the right-handed trefoil $3_1^+$, and
\[
3_1^+\in S(4_1).
\]
Mirroring the entire picture replaces $F^{-,+}$ by $F^{+,-}$ and gives the left-handed
trefoil.  Since $4_1$ is amphichiral, this also proves $3_1^-\in S(4_1)$.
\end{proof}

\begin{corollary}\label{cor:3_1-4_1-not-amicable}
Neither $3_1^+$ nor $3_1^-$ is amicable with $4_1$.
\end{corollary}

\begin{proof}
Proposition~\ref{prop:3_1-in-S4_1} gives realizability from $4_1$ to either trefoil.
Corollary~\ref{cor:4_1_not_in_S3_1} gives the failure in the reverse direction.
\end{proof}

\begin{figure}[b]
\centering
\begin{tikzpicture}[scale=1.2]
  \node[knot] (a) at (0,0) {$3_1^+$};
  \node[knot] (b) at (3,1.2) {$8_{19}^+$};
  \node[knot] (c) at (3,-1.2) {$4_1$};

  \draw[<->,thick] (a) -- node[above,info,yshift=2.0mm] {amicable} (b);
  \draw[->,thick] (c) -- node[right,info,yshift=1.5mm] {$3_1^+\in S(4_1)$} (a);
  \draw[thick,dashed] (a) -- node[below left,info] {$4_1\notin S(3_1^+)$} (c);
\end{tikzpicture}
\caption{The right-handed knots $3_1^+$ and $8_{19}^+$ are amicable, whereas
$3_1^+$ and $4_1$ are not.  Mirroring gives the corresponding left-handed statements.}
\label{fig:amicable-examples}
\end{figure}

\section{Final Questions}\label{sec:final-questions}

The examples above suggest that the symmetric relation of amicability deserves further study.
We conclude by returning to Questions~\ref{que:proper-amicable}
and~\ref{que:amicability-connected}.  The first asks whether every knot has a
proper amicable knot, while the second asks whether the amicability graph is
connected.

\begin{remark}
Corollary~\ref{cor:stabilized-amicability} shows that the connected-sum operation interacts
naturally with the amicability graph: any edge \(K\!-\!J\) gives rise to edges
\[
K\#L \;-\; J\#L
\]
for all knots \(L\).
Thus connected sum supplies a systematic way of constructing large families of vertices in the
same connected component once one amicable pair is known.
At the same time, this stabilization phenomenon does not obviously connect unrelated components,
so Question~\ref{que:amicability-connected} remains open.
\end{remark}

\end{document}